\documentclass{article}
\usepackage{latexsym}
\usepackage{leftidx}
\usepackage{amsmath,amsthm, color, pdfpages}
\usepackage{enumerate}
\usepackage{amssymb}
\usepackage{upgreek}
\usepackage[margin=1.1in,dvips]{geometry}

\def\f12{\frac 1 2}

\def\ga{\gamma}

\def\ep{\epsilon}

\def\si{\sigma}

\def\om{\omega}
\def\Om{\Omega}

\def\S{\mathcal{S}} % 2-sphere

\def\Lb{\underline{L}}

\def\pa{\partial}
\def\les{\lesssim}
\def\B{\mathcal{B}}% the initial hypersurface
\def\cL{{\mathcal L}}

\def\f12{\frac 1 2}

\newcommand{\vol}{\textnormal{vol}}

\newcommand{\nabb}{\mbox{$\nabla \mkern-13mu /$\,}}

\newtheorem{Thm}{Theorem}[section]

\newtheorem{Prop}{Proposition}[section]

\newtheorem{Lem}{Lemma}[section]

\newtheorem{Remark}{Remark}[section]
\theoremstyle{definition}

\begin{document}

\title{Uniform bound for solutions of semilinear wave equations in $\mathbb{R}^{1+3}$}
\date{}
\author{Shiwu Yang}
\maketitle
\begin{abstract}
We prove that solution of defocusing semilinear wave equation in $\mathbb{R}^{1+3}$ with pure power nonlinearity is uniformly bounded for all $\frac{3}{2}<p\leq 2$ with sufficiently smooth and localized data. The result relies on the $r$-weighted energy estimate originally introduced by Dafermos and Rodnianski. This appears to be the first result regarding the global asymptotic property for the solution with small power $p$ under 2.
\end{abstract}

\section{Introduction}
In this paper, we continue our study on the global pointwise behaviors for solutions to the energy subcritical defocusing semilinear wave equations
\begin{equation}
  \label{eq:NLW:semi:3d}
  \Box\phi=(-\pa_{tt}+\Delta)\phi=|\phi|^{p-1}\phi,\quad \phi(0, x)=\phi_0(x),\quad \pa_t\phi(0, x)=\phi_1(x)
\end{equation}
in $\mathbb{R}^{1+3}$ with small $p$ in the range $1<p\leq 2$.
For constant $\ga$, define the weighted energy norm of the initial data
\begin{align*}
  \mathcal{E}_{k,\ga}[\phi]=\sum\limits_{l\leq k}\int_{\mathbb{R}^3}(1+|x|)^{\ga+2l}(|\nabla^{l+1}\phi_0|^2+|\nabla^l \phi_1|^2)+(1+|x|)^{\ga}|\phi_0|^{p+1}dx.
\end{align*}
We prove in this paper that
\begin{Thm}
\label{thm:main}
Consider the defocusing semilinear wave equation \eqref{eq:NLW:semi:3d} with initial data $(\phi_0, \phi_1)$ such that $\mathcal{E}_{1, 0}[\phi]$ and $\mathcal{E}_{0, p-1}[\phi]$ are finite. Then for all $\frac{3}{2}<p\leq 2$,
the solution $\phi$ to the equation \eqref{eq:NLW:semi:3d} exists globally in time and is uniformly bounded in the following sense
\begin{equation}
\label{eq:phi:pt:Br:largep}
|\phi(t, x)|\leq C(\sqrt{\mathcal{E}_{1, 0}[\phi]}+(\mathcal{E}_{0, p-1}[\phi])^{\frac{p}{p+1}}),\quad \forall (t, x)\in \mathbb{R}^{1+3}
\end{equation}
for some constant $C$ depending only on $p$.
\end{Thm}
We give two remarks.
\begin{Remark}
  In view of the energy conservation, one can easily conclude that the solution grows at most polynomially in time $t$ with rate relying on the power $p$. The theorem firstly improves this growth to be uniformly bounded.
\end{Remark}
\begin{Remark}
Our proof also implies that the solution decays weakly in the spatial variable $x$.
\end{Remark}

For initial data and $p$ required in the above theorem, the solution to \eqref{eq:NLW:semi:3d} exists globally as shown by Ginibre-Velo, who derived global well-posedness in energy space for all $1<p<5$ in \cite{velo85:global:sol:NLW}. Related works regarding the global regularity could be found, for example, in
\cite{Jorgens61:energysub:NLW:lowerd}, \cite{Brenner79:globalregularity:NW}, \cite{Pecher76:NLW:global}, \cite{segal63:semigroup}, \cite{Vonwahl75:NW}, \cite{Velo89:globalsolution:NLW} and references therein. Asymptotic behavior of this global solution mainly concerns two types of questions: The first type is the problem of scattering, namely comparing the solution with linear solution as time goes to infinity in certain Sobolev spaces, like the energy space, the critical Sobolev space and the conformal energy space (weighted energy space). We refer to the latest work \cite{yang:scattering:NLW} for more detailed discussions.

Motivated by the scattering problem which requires a priori uniform spacetime bound for the solution, one can alternatively study the pointwise decay properties of the solution, initiated by the early work of
 Strauss \cite{Strauss:NLW:decay}, followed by extensions in \cite{vonWahl72:decay:NLW:super}, \cite{Roger:3DNLW:symmetry},  \cite{Bieli:3DNLW}, \cite{Pecher82:decay:3d}, \cite{yang:NLW:ptdecay:3D}. In any case, it is crucial to control the nonlinearity which is equivalent to certain decay properties of the solution. In particular, the larger $p$ leads to faster decay of the nonlinearity. Indeed scattering in energy space holds when $2.3542<p<5$, which follows from the pointwise decay estimates of the solution available only when $2<p<5$ (see the author's work \cite{yang:NLW:ptdecay:3D}).

The approach to study the asymptotic behavior of solutions to \eqref{eq:NLW:semi:3d} before the works \cite{yang:scattering:NLW}, \cite{yang:NLW:ptdecay:3D} relied on the time decay of the potential energy
\begin{align*}
%\label{eq:timedecay:3D:Pecher}
  \int_{\mathbb{R}^3}|\phi|^{p+1}dx\leq C (1+t)^{\max\{4-2p, -2\}},\quad 1<p<5,
\end{align*}
obtained by using the conformal Killing vector field $t^2\pa_t+r^2 \pa_r$ ($r=|x|$) as multiplier
 (see\cite{Pecher82:decay:3d}, \cite{Velo87:decay:NLW}). It is obvious that the potential energy decays faster for larger $p$. This decay estimate was used by Pecher \cite{Pecher82:decay:3d} to derive the pointwise decay estimate for the solution when $p>\frac{1+\sqrt{13}}{2}$ with a corollary that the solution scatters in energy space for $p>2.7005$.

 By using the vector field method originally introduced by Dafermos-Rodnianski \cite{newapp}, the author in \cite{yang:scattering:NLW}, \cite{yang:NLW:ptdecay:3D} was able to extend the above asymptotic behaviors to $2<p<5$.  This method relies on the $r$-weighted energy estimate
  \begin{align}
  \label{eq:rp:bd}
    \iint_{\mathbb{R}^{3+1}}\frac{p-1-\ga}{p+1}r^{\ga-1}|\phi|^{p+1}dxdt\leq C \mathcal{E}_{0, \ga}[\phi]
  \end{align}
   derived by using the vector field $r^{\ga}(\pa_t+\pa_r)$ as multiplier with the restriction $0<\ga<p-1$. Combining this estimate with an integrated local energy estimate of the solution, one can obtain the following uniform weighted spacetime bound
\begin{align*}
  \iint_{\mathbb{R}^{3+1}}(1+t+|x|)^{\ga-1-\ep}|\phi|^{p+1}dxdt\leq C\mathcal{E}_{0, \ga}[\phi],\quad \forall \ep>0
\end{align*}
for the case when $\ga>1$, which forces $p>2$. This uniform bound allows us to use the vector field
\begin{align*}
  X^{\ga}=u_+^{\ga}(\pa_t-\pa_r)+v_+^{\ga}(\pa_t+\pa_r),\quad v_+=\sqrt{1+(t+|x|)^2},\quad u_+=\sqrt{1+u^2}
\end{align*}
as multiplier applied to any backward light cone, which then leads to the pointwise decay estimate for the solution (see details in \cite{yang:NLW:ptdecay:3D}).

The above argument fails for the case when $p\leq 2$ for the reason that in this case, $\ga\leq 1$ and the spacetime bound \eqref{eq:rp:bd} is not sufficient to control the spacetime error term generated by using the vector field $X^{\ga}$ as multiplier. We also note that the above time decay of potential energy is only available for the case when $p> 2$. The key observation for the small power $p$ case is that instead of using the vector field $X^{\gamma }$ as multiplier, we apply the vector field $r^{\ga}(\partial_t+\partial_r)$ to regions bounded by the backward light cone $\mathcal{N}^{-}(q)$ emanating from the point $q\in\mathbb{R}^{1+3}$ to obtain the following weighted energy estimate 
 \begin{align*}
  \int_{\mathcal{N}^{-}(q)}\big((1+\frac{x\cdot(x-x_0)}{|x||x-x_0|})r^{\ga}+1\big)|\phi|^{p+1}d\sigma
\leq C \mathcal{E}_{0, \ga}[\phi],\quad 0<\ga<p-1.
\end{align*}
This estimate is sufficient to conclude the main theorem that the solution is uniformly bounded for the case when $p>\frac{3}{2}$.

\section{Preliminaries and notations}
\label{sec:notation}

We use the standard polar local coordinate system $(t, r,
\om)$ of Minkowski space as well as the null coordinates $u=\frac{t-r}{2}$, $v=\frac{t+r}{2}$, in which $\om=\frac{x}{|x|}$ is the coordinate of the unit sphere.
Introduce a null frame $\{L, \Lb, e_1, e_2\}$ such that
\[
L=\pa_v=\pa_t+\pa_r,\quad \Lb=\pa_u=\pa_t-\pa_r
\]
and $\{e_1, e_2\}$ an orthonormal basis of the sphere with
constant radius $r$.
Let $\nabla$ be the shorthand for $(\pa_{x^1}, \pa_{x^2}, \pa_{x^3})$.

For any point $q=(t_0, x_0)\in \mathbb{R}^{3+1}$ and $r>0$, denote $\B_q(r)$ as the 3-dimensional ball at time $t_0$ with radius $r$ centered at $q$, that is,
\begin{align*}
  \B_q( r)=\{(t,x)| t=t_0, |x-x_0|\leq r\}.
\end{align*}
 The boundary of $\B_q( r)$ is the 2-sphere $\S_q(r)$. Without loss of generality, we only consider the solution in the future $t\geq 0$. Define the past null cone at $q$ as $\mathcal{N}^{-}(q)$, that is,
 \begin{align*}
   \mathcal{N}^{-}(q):=\{(t, x)| |t-t_0|=|x-x_0|,\quad t\geq 0\}.
 \end{align*}
 The region enclosed by this cone is the past of the point $q$ and we denote it as $\mathcal{J}^{-}(q)$, that is,
   \begin{align*}
   \mathcal{J}^{-}(q):=\{(t, x)| |x-x_0|\leq |t-t_0|,\quad t\geq 0\}.
 \end{align*}
Additional to the standard coordinates $(t, x)$ as well as the associated polar coordinates, let $(\tilde{t}, \tilde{x})$ be the new coordinates centered at the point $q=(t_0, x_0)$
\[
\tilde{t}=t-t_0,\quad \tilde{x}=x-x_0,\quad \tilde{r}=|\tilde{x}|,\quad \tilde{\om}=\frac{\tilde{x}}{|\tilde{x}|},\quad \tilde{u}=\f12 (\tilde{t}-\tilde{r}),\quad \tilde{v}=\f12(\tilde{t}+\tilde{r}).
\]
We also have the associated null frame $\{\tilde{L}, \tilde{\Lb}, \tilde{e}_1, \tilde{e}_2\}$. Under this new coordinates, the past null cone $\mathcal{N}^{-}(q)$ can be characterized by $\{\tilde{v}=0\}\cap\{0\leq t\leq t_0\}$. Through out this paper, the coordinates $(\tilde{t}, \tilde{x})$ are always referred to be the translated ones centered at the point  $q=(t_0, x_0)$ unless it is clearly emphasized.

Finally we make a convention that $A\les B$ means there exists a constant $C$, depending only on $p$ such that $A\leq CB$.

\section{A uniform $r$-weighted energy estimate through backward light cones}
Our goal is to investigate the pointwise decay properties for solutions of \eqref{eq:NLW:semi:3d} for small power $p$ such that $1<p\leq 2$. Following the method introduced in \cite{yang:NLW:ptdecay:3D}, we first derive a uniform weighted energy estimate through backward light cones.
\begin{Prop}
\label{prop:EF:cone:NW:3d:smallp}
Assume that $1<p\leq 2$.
Let $q=(t_0, x_0)$ be any point in $\mathbb{R}^{3+1}$. Then for solution $\phi$ of the nonlinear wave equation \eqref{eq:NLW:semi:3d} and for all $0\leq \ga\leq   p-1$, we have the following uniform bound
\begin{equation}
\label{eq:Eflux:ex:EF}
\begin{split}
&\int_{\mathcal{N}^{-}(q)}((1+\tau)r^{\ga}+1)|\phi|^{p+1} d\si %+\int_{\mathcal{J}^{-}(q)}r^{\ga-3}(\ga|L(r\phi)|^2+(2-\ga)|\nabb(r\phi)|^2)% \tilde{r}^2 d\tilde{u}d\tilde{\om}
\leq C \mathcal{E}_{0, \ga}[\phi]
\end{split}
\end{equation}
for some constant $C$ depending only on $p$. Here $d\si$ is the surface measure, $\tau=\om\cdot \tilde{\om}$ and the tilde components are measured under the coordinates $(\tilde{t}, \tilde{x})$ centered at the point $q=(t_0, x_0)$.
\end{Prop}
\begin{proof}
The proof goes similar to those in \cite{yang:NLW:ptdecay:3D}, \cite{yang:scattering:NLW}. We choose vector fields as in \cite{yang:scattering:NLW} but apply them to the cone $\mathcal{J}^{-}(q)$ as in \cite{yang:NLW:ptdecay:3D}. We repeat the proof here but may skip some detailed computations which have already done in \cite{yang:NLW:ptdecay:3D}, \cite{yang:scattering:NLW}.

Let's first review the vector field method. Recall the energy momentum tensor for the scalar field $\phi$
\begin{align*}
  T[\phi]_{\mu\nu}=\pa_{\mu}\phi\pa_{\nu}\phi-\f12 m_{\mu\nu}(\pa^\ga \phi \pa_\ga\phi+\frac{2}{p+1} |\phi|^{p+1}),
\end{align*}
where $m_{\mu\nu}$ is the flat Minkowski metric on $\mathbb{R}^{1+d}$. Then we can compute that
\begin{align*}
  \pa^\mu T[\phi]_{\mu\nu}=&(\Box\phi-  |\phi|^{p-1}\phi)\pa_\nu\phi.
\end{align*}
Now for any vector fields $X$, $Y$ and any function $\chi$, define the current
\begin{equation*}
J^{X, Y, \chi}_\mu[\phi]=T[\phi]_{\mu\nu}X^\nu -
\f12\pa_{\mu}\chi \cdot|\phi|^2 + \f12 \chi\pa_{\mu}|\phi|^2+Y_\mu.
\end{equation*}
Then for solution $\phi$ of equation \eqref{eq:NLW:semi:3d}, we have the energy identity
\begin{equation*}
%\label{eq:energy:id}
\iint_{\mathcal{D}}\pa^\mu  J^{X,Y,\chi}_\mu[\phi] d\vol =\iint_{\mathcal{D}}div(Y)+ T[\phi]^{\mu\nu}\pi^X_{\mu\nu}+
\chi \pa_\mu\phi\pa^\mu\phi -\f12\Box\chi\cdot|\phi|^2 +\chi \phi\Box\phi d\vol
\end{equation*}
for any domain $\mathcal{D}$ in $\mathbb{R}^{3+1}$. Here $\pi^X=\f12 \cL_X m$  is the deformation tensor for the vector field $X$.

In the above energy identity, choose the vector fields $X$, $Y$ and the function $\chi$ as follows:
\[
X=r^{\gamma} L,\quad Y=\frac{1}{2}\ga r^{\ga-2}|\phi|^2 L, \quad \chi=r^{\ga-1},\quad \forall 0\leq \ga\leq 2.
\]
The computations in \cite{yang:scattering:NLW} show that
%\[
%\nabla_{L}X=\gamma r^{\gamma-1}L,\quad \nabla_{\Lb}X= -\ga r^{\ga-1}L,\quad \nabla_{e_A}X=r^{\gamma-1} e_A.
%\]
%In particular the non-vanishing components of the deformation tensor $\pi_{\mu\nu}^X$ are
%\[
%\pi^X_{L\Lb}=-\gamma r^{\gamma-1},\quad \pi^X_{\Lb\Lb}=2\gamma r^{\gamma-1},\quad \pi^X_{e_A e_A}=r^{\ga-1}.
%\]
%Denote $\psi=r \phi$.
%Then we can compute that
\begin{align*}
&div(Y)+T[\phi]^{\mu\nu}\pi^X_{\mu\nu}+
\chi \pa_\mu\phi \pa^\mu\phi  +\chi\phi\Box\phi-\f12 \Box\chi |\phi|^2\\
%&=-\f12\ga r^{\gamma-1}(|\nabb\phi|^2+\frac{2}{p+1}|\phi|^{p+1})+r^{\ga-1}(|\nabb\phi|^2- \pa^\mu \phi \pa_\mu\phi-\frac{2}{p+1}|\phi|^{p+1})\\
%&\quad +\f12\ga r^{\ga-1}|L\phi|^2+
%\chi \pa_\mu\phi \pa^\mu\phi  +\chi|\phi|^{p+1}-\f12 \Box\chi |\phi|^2+div(Y)\\
%&=\f12 r^{\ga-1}(\ga|L\phi|^2+(2-\ga)|\nabb\phi|^2 )+\frac{p-1-\ga}{p+1} r^{\ga-1}|\phi|^{p+1}\\
%&\quad -\frac{(\ga-1) \ga }{2}r^{\ga-3}|\phi|^2+\frac{ \ga}{2}(L(r^{\ga-2}|\phi|^2)+2 r^{\ga-3}|\phi|^2)\\
&=\f12 r^{\ga-3}(\ga|L(r\phi)|^2+(2-\ga)|\nabb(r\phi)|^2 )+\frac{p-1-\ga }{p+1} r^{\ga-1}|\phi|^{p+1} .
\end{align*}
For the case when $0\leq  \ga \leq p-1\leq 2$, this term is nonnegative.

Let the domain $\mathcal{D}$ be $\mathcal{J}^{-}(q)$ with boundary $\B_{(0, x_0)}(t_0)\cup \mathcal{N}^{-}(q)$. By using Stokes' formula, the left hand side of the above energy identity consists of the integral on the initial hypersurface $\B_{(0, x_0)}(t_0)$ and on the backward light cone $\mathcal{N}^{-}(q)$.
For the integral on $\B_{(0, x_0)}(t_0)$, recall from \cite{yang:scattering:NLW} that
\begin{align}
\label{eq:PWE:ex:bxt0}
 \int_{\B_{(0, x_0)}(t_0)} i_{J^{X, Y,\chi}[\phi]}d\vol %&= \int_{\B_{x_0}(t_0)} (J^{X, Y,\chi}[\phi])^{0}dx= -( T[\phi]_{0 \nu}X^\nu -
%\f12 \pa_t\chi |\phi|^2 + \f12 \chi\cdot \pa_t |\phi|^2 +Y_0)  dx\\
&=\f12 \int_{\B_{(0, x_0)}(t_0)}  r^\ga(  r^{-2}|L(r\phi)|^2 +|\nabb\phi|^2+\frac{2}{p+1}|\phi|^{p+1})  -\pa_r  (  r^{1+\ga} |\phi|^2 )r^{-2}dx.
\end{align}
 For the boundary integral on the backward light cone $\mathcal{N}^{-}(q)$, as in \cite{yang:NLW:ptdecay:3D}, we compute the explicit form under the coordinates centered at the point $q=(t_0, x_0)$.
 Recall the volume form
\[
d\vol=dxdt=d\tilde{x}d\tilde{t}=2\tilde{r}^2 d\tilde{v}d\tilde{u}d\tilde{\om}.
\]
Under these new coordinates $(\tilde{t}, \tilde{x})$, we can compute that
\begin{align*}
-i_{J^{X, Y,\chi}[\phi]}d\vol=J_{\tilde{\Lb}}^{X, Y,\chi}[\phi]\tilde{r}^2d\tilde{u}d\tilde{\om}= ( T[\phi]_{\tilde{\Lb}\nu}X^\nu -
\f12(\tilde{\Lb}\chi) |\phi|^2 + \f12 \chi\cdot\tilde{\Lb}|\phi|^2 +Y_{\tilde{\Lb}}) \tilde{r}^2d\tilde{u}d\tilde{\om}.
\end{align*}
For the main quadratic terms, we have
\begin{align*}
T[\phi]_{\tilde{\Lb}\nu}X^\nu =T[\phi]_{\tilde{\Lb}\tilde{\Lb}}X^{\tilde{\Lb}}+T[\phi]_{\tilde{\Lb}\tilde{L}}X^{\tilde{L}}+T[\phi]_{\tilde{\Lb}\tilde{e}_i}X^{\tilde{e}_i}.
\end{align*}
Now we need to write the vector field $X$ under the new null frame $\{\tilde{L}, \tilde{\Lb}, \tilde{e}_1, \tilde{e}_2\}$ centered at the point $q$. Note that
\begin{align*}
\pa_r=\om \cdot \nabla=\om \cdot \tilde{\nabla}=\om\cdot \tilde{\om}\pa_{\tilde{r}}+ \om\cdot (\tilde{\nabla}-\tilde{\om}\pa_{\tilde{r}}).
\end{align*}
Then we have
\begin{align*}
X=r^\gamma (\pa_t+\pa_r)
&=r^\ga (\pa_{\tilde{t}}+\om\cdot \tilde{\om}\pa_{\tilde{r}}+ \om\cdot \tilde{\nabb})=\f12 r^\ga(1+\om\cdot \tilde{\om})\tilde{L}+\f12 r^\ga(1-\om\cdot \tilde{\om}) \tilde{\Lb}+r^\ga \om\cdot \tilde{\nabb}.
\end{align*}
Here $\tilde{\nabb}=\tilde{\nabla}-\tilde{\om}\pa_{\tilde{r}}$. Denote $\tau=\om\cdot \tilde{\om}$. Then we can compute the quadratic terms
\begin{align*}
T[\phi]_{\tilde{\Lb}\nu}X^\nu %-m_{\tilde{\Lb}\nu}X^\nu \frac{1}{p+1}\La^{3-p}|\phi|^{p+1}\\
=& \f12 (1-\tau)r^\ga |{\tilde{\Lb}}\phi|^2 + \f12 (1+\tau)r^\ga (|\tilde{\nabb}\phi|^2+\frac{2}{p+1}|\phi|^{p+1})+r^\ga  ({\tilde{\Lb}}\phi) (\om\cdot \tilde{\nabb})\phi.
\end{align*}
These terms are nonnegative. Indeed note that
\begin{align*}
\tilde{\Lb}(r)=-\tilde{\om}_i\pa_i(r)=-\tilde{\om}\cdot \om =-\tau,\quad \tilde{\nabb}(r)=(\tilde{\nabla}-\tilde{\om}\pa_{\tilde{r}})(r)=\om-\tilde{\om}\tau.
\end{align*}
Therefore we can write
\begin{align*}
&-\f12 r^2 (\tilde{\Lb}{\chi})|\phi|^2+\f12 r^2\chi \tilde{\Lb}|\phi|^2+r^2 Y_{\tilde{\Lb}}=r^{\ga}( {\tilde{\Lb}}(r\phi)+\tau \phi) \phi+\f12\tau (\ga-1)r^{\ga}|\phi|^2-\f12\ga (1+\tau)r^{\ga}|\phi|^2,\\
&r^2|\tilde{\Lb}\phi|^2=|{\tilde{\Lb}}(r\phi)-\tilde{\Lb}(r)\phi|^2=|{\tilde{\Lb}}(r\phi)|^2+\tau^2|\phi|^2+2 {\tilde{\Lb}}(r\phi) \tau\phi,\\
& r^2|\tilde{\nabb}\phi|^2%=|\tilde{\D}(r\phi)-\tilde{\D}(r)\phi|^2
=|\tilde{\nabb}(r\phi)|^2+(1-\tau^2)|\phi|^2-2(\om-\tilde{\om}\tau)\cdot\tilde{\nabb}(r\phi) \phi,\\
& r^2 ({\tilde{\Lb}}\phi)  (\om\cdot \tilde{\nabb})\phi={\tilde{\Lb}}(r\phi) (\om \cdot \tilde{\nabb})(r\phi)-\tau(1-\tau^2)|\phi|^2+\phi \tau(\om\cdot \tilde{\nabb})(r\phi) -(1-\tau^2){\tilde{\Lb}}(r\phi)\phi.
\end{align*}
Notice that
\begin{align*}
  |(\om\cdot \tilde{\nabb})(r\phi)|=|(\om\times \tilde{\om}\cdot \tilde{\nabb})(r\phi)|=\sqrt{1-\tau^2}|\tilde{\nabb}(r\phi)|.
\end{align*}
In particular the quadratic terms are nonnegative
\begin{align*}
&\f12 (1-\tau)r^\ga |\tilde{\Lb}(r\phi)|^2+\f12(1+\tau)r^\ga |\tilde{\nabb}(r\phi)|^2+r^\ga {\tilde{\Lb}}(r\phi) (\om \cdot \tilde{\nabb})(r\phi)\geq 0.
\end{align*}
For the lower order terms other terms, we compute that
\begin{align*}
&\f12 r^{\ga}(1-\tau)(\tau^2|\phi|^2+2{\tilde{\Lb}}(r\phi)\tau\phi )+r^{\ga}( {\tilde{\Lb}}(r\phi)+\tau \phi) \phi+\f12\tau (\ga-1)r^{\ga}|\phi|^2-\f12\ga (1+\tau)r^{\ga}|\phi|^2\\
&+\f12 r^{\ga}(1+\tau) \big((1-\tau^2)|\phi|^2-2(\om-\tilde{\om}\tau)\tilde{\nabb}(r\phi) \phi\big)\\
&+r^{\ga}\big(-\tau(1-\tau^2)|\phi|^2+\phi \tau(\om\cdot \tilde{\nabb})(r\phi)-\phi (1-\tau^2){\tilde{\Lb}}(r\phi) \big)\\
%&+(v_*^\ga-u_*^\ga)<\tau\om\cdot \tilde{\D}(r\phi)-(1-\tau^2)D_{\tilde{\Lb}}(r\phi),\phi>\\
&=\f12(1-\ga)r^{\ga}|\phi|^2%-2(v_*^\ga+u_*^\ga)<\om\cdot \tilde{\D}(r\phi), \phi>
+r^{\ga}(\tau {\tilde{\Lb}}-\om\cdot \tilde{\nabb})(r\phi) \phi\\
%&=(-r\tilde{\Lb}(r\chi)+v_*^\ga+u_*^\ga)|\phi|^2-r^{-1}(v_*^\ga+u_*^\ga) (\om\cdot \tilde{\nabb})(|r\phi|^2)+r^{-1}(v_*^\ga+u_*^\ga)\tilde{\Lb}(|r\phi|^2)\\
&=-\f12 r^2\tilde{r}^{-1} \tilde{\Om}_{ij}(r^{\ga-1} \om_j\tilde{\om}_i |\phi|^2)+\f12 \tilde{r}^{-2}r^2\tilde{\Lb}(r^{\ga-1}\tau \tilde{r}^2 |\phi|^2)\\
&+\f12(1-\ga)r^\ga|\phi|^2-\f12 \tilde{r}^{-2}r^2\tilde{\Lb}(r^{\ga-3}\tau\tilde{r}^2) |r\phi|^2+\f12 r^2 \tilde{r}^{-1}\tilde{\Om}_{ij}(r^{\ga-3} \om_j\tilde{\om}_i) |r\phi|^2.
\end{align*}
Here $\tilde{\Om}_{ij}=\tilde{x}_i\tilde{\pa}_{j}-\tilde{x}_j\tilde{\pa}_{{i}}=\tilde{r}(\tilde{\om}_i\tilde{\pa}_j-\tilde{\om}_j\tilde{\pa}_i)$ and we have omitted the summation sigh for repeated indices $i$, $j$ for simplicity. By computations, note that
\begin{align*}
&\tilde{r}^{-1}\tilde{\Om}_{ij}(r^{-3}\om_j\tilde{\om}_i)=-2r^{\ga-4}(1-2\tau^2)-2\tau \tilde{r}^{-1}r^{\ga-3}+\ga r^{\ga-4}(1-\tau^2),\\
&\tilde{r}^{-2}r^{4-\ga}\tilde{\Lb}(r^{\ga-3}\tilde{r}^2\tau)=4\tau^2-1-2r\tilde{r}^{-1}\tau-\ga\tau^2.
\end{align*}
Thus the last line in the previous equality vanishes
\begin{align*}
  &\f12(1-\ga)r^{\ga}|\phi|^2-\f12 \tilde{r}^{-2}r^2\tilde{\Lb}(r^{\ga-3}\tau\tilde{r}^2) |r\phi|^2+\f12 r^2 \tilde{r}^{-1}\tilde{\Om}_{ij}(r^{\ga-3} \om_j\tilde{\om}_i) |r\phi|^2\\
  &=\f12 r^{\ga}|\phi|^2\left(1-\ga-4\tau^2+1+2r\tilde{r}^{-1}\tau+\ga\tau^2-2 (1-2\tau^2)-2\tau \tilde{r}^{-1}r +\ga  (1-\tau^2) \right)\\
  &=0.
\end{align*}
By using integration by parts on the backward light cone $\mathcal{N}^{-}(q)$, the integral of the second last line is
\begin{align*}
  &\int_{\mathcal{N}^{-}(q)}\big(-\f12 r^2\tilde{r}^{-1} \tilde{\Om}_{ij}(r^{\ga-3} \om_j\tilde{\om}_i |r\phi|^2)+\f12 \tilde{r}^{-2}r^2\tilde{\Lb}(  r^{\ga-1}\tau\tilde{r}^2 |\phi|^2)\big)r^{-2}\tilde{r}^2 d\tilde{u}d\tilde{\om}\\
  &= \f12 \int_{\S_{(0, x_0)}(t_0)}r^{\ga-1}\tau \tilde{r}^2 |\phi|^2d\tilde{\om}.
\end{align*}
The above computations show that the quadratic terms are nonnegative and the lower order terms are equal to the above integral on the 2-sphere on the initial hypersurface, that is,
\begin{equation*}
%\label{eq:EST:Nq:comp:ex}
\begin{split}
&-\int_{\mathcal{N}^{-}(q)}i_{J^{X, Y, \chi}[\phi]}d\vol+ \f12 \int_{\S_{(0, x_0)}(t_0) }r^{\ga-1}\tau \tilde{r}^2 |\phi|^2d\tilde{\om}\geq \int_{\mathcal{N}^{-}(q)} \frac{1+\tau}{p+1} r^{\ga}|\phi|^{p+1}  \tilde{r}^2 d\tilde{u}d\tilde{\om}.
\end{split}
\end{equation*}
On the other hand for the case when $0\leq \ga \leq p-1<2$, the bulk integral on the right hand side of the energy identity is nonnegative, that is,
\begin{align*}
  \int_{\mathcal{N}^{-}(q)}i_{J^{X, Y, \chi}[\phi]}d\vol+\int_{\B_{(0, x_0)}(t_0)}i_{J^{X, Y, \chi}[\phi]}d\vol\geq 0.
\end{align*}
Adding this estimate to the previous inequality and in view of the expression \eqref{eq:PWE:ex:bxt0}, we conclude that
\begin{align*}
   &\int_{\B_{(0, x_0)}(t_0)}  r^\ga(  r^{-2}|L(r\phi)|^2 +|\nabb\phi|^2+\frac{2}{p+1}|\phi|^{p+1})  -\pa_r  (  r^{1+\ga} |\phi|^2 )r^{-2}dx\\
   &+  \int_{\S_{(0, x_0)}(t_0) }r^{\ga-1}\tau \tilde{r}^2 |\phi|^2d\tilde{\om}\geq 2\int_{\mathcal{N}^{-}(q)} \frac{1+\tau}{p+1} r^{\ga}|\phi|^{p+1}  \tilde{r}^2 d\tilde{u}d\tilde{\om}.
\end{align*}
Note that
\begin{align*}
  &\int_{\B_{(0, x_0)}(t_0)}   \pa_r  (  r^{1+\ga} |\phi|^2 )r^{-2}dx=\int_{\B_{(0, x_0)}(t_0)}  \textnormal{div} (\om r^{\ga-1}|\phi|^2)dx%=\int_{\B_{(0, x_0)}(t_0)}  \textnormal{div} (\om r^{\ga-1}|\phi|^2)d\tilde{x}
  =\int_{\S_{(0, x_0)}(t_0) }r^{\ga-1}\tau \tilde{r}^2 |\phi|^2d\tilde{\om}.
\end{align*}
 We therefore derive from the previous inequality that
 \begin{align*}
   2\int_{\mathcal{N}^{-}(q)} \frac{1+\tau}{p+1} r^{\ga}|\phi|^{p+1}  d\sigma\leq  \int_{\B_{(0, x_0)}(t_0)}  r^\ga(  r^{-2}|L(r\phi)|^2 +|\nabb\phi|^2+\frac{2}{p+1}|\phi|^{p+1})  dx \leq \mathcal{E}_{0, \ga}[\phi].
 \end{align*}
The uniform bound \eqref{eq:Eflux:ex:EF} of the Proposition then follows by the standard energy estimates obtained by using the vector field $\pa_t$ as multiplier.
\end{proof}

\section{The uniform pointwise bound of the solution}
In this section, we make use of the weighted energy flux bound derived in the previous section to investigate the asymptotic behaviour of the solution. The idea is to use the uniform weighted energy estimate to control the nonlinearity directly. For this purpose, we need the following technical integration lemma.

% we are trying to prove a technical integration lemma.
\begin{Lem}
\label{lem:integration:smallp}
  Assume $0\leq \ga< p-1$. Fix $q=(t_0, x_0)$ in $\mathbb{R}^{3+1}$. For the $2$-sphere $\S_{(t_0-\tilde{r}, x_0)}(\tilde{r})$ on the backward light cone $\mathcal{N}^{-}(q)$, we have
  \begin{equation}
    \label{eq:integration:ex:ab}
    \begin{split}
    &\int_{\S_{(t_0-\tilde{r}, x_0)}(\tilde{r})} ((1+\tau)r^{\ga}+1)^{-p} d\tilde{\om} \leq C (1+r_0+\tilde{r})^{-\ga}
    \end{split}
  \end{equation}
  for some constant $C$ depending only on $p$ and $\ga$.
  Here $\tau=\om\cdot \tilde{\om}$, $r_0=|x_0|$ and $0\leq \tilde{r}\leq t_0 $.
\end{Lem}
\begin{proof}
During the proof we also let the implicit constant in $\les$ rely on $\ga$.
Denote $s=-\om_0\cdot \tilde{\om}$ and $\om_0=r_0^{-1}x_0$. By definition, we have
  \begin{align*}
    &r^2=|x_0+\tilde{x}|^2=\tilde{r}^2+r_0^2-2r_0\tilde{r}s=(\tilde{r}-r_0s)^2+(1-s^2)r_0^2,\\
    &(1+\tau)r=r+r\om\cdot \tilde{\om}=r+(\tilde{x}+x_0)\cdot \tilde{\om}=r+\tilde{r}-r_0s.
  \end{align*}
We can write the integral as
  \begin{align*}
    &\int_{\S_{(t_0-\tilde{r}, x_0)}(\tilde{r})} ((1+\tau)r^{\ga}+1)^{-p} d\tilde{\om} =2\pi\int_{-1}^1 (r^{\ga-1}(r+\tilde{r}-r_0s)+1)^{-p}ds.
  \end{align*}
First we consider the case when the point $q$ locates in the exterior region, that is, $t_0\leq r_0$. The case when $t_0+r_0\leq 10$ is trivial. Thus in the following we always assume that $t_0+r_0\geq 10$. In particular in the exterior region, $r_0\geq 5$.
For the integral on $s\leq 0$, we trivially bound that
  \begin{align*}
    r^{\ga-1}(r+\tilde{r}-r_0s)\geq r_0^{\ga}.
  \end{align*}
  Therefore we can estimate that
  \begin{align*}
    \int_{-1}^0  (r^{\ga-1}(r+\tilde{r}-r_0s)+1)^{-p}ds\leq r_0^{-p\ga}.
  \end{align*}
  Define $s_0=1-(1-\tilde{r}r_0^{-1})^2$. On the interval $[0, s_0]$, note that
  \begin{align*}
    \sqrt{1-s} \ r_0\geq r_0-\tilde{r}.
  \end{align*}
Since $\tilde{r}\leq t_0\leq r_0$ and $0\leq s\leq s_0\leq 1$, we have
  \begin{align*}
    \tilde{r}-r_0s \leq r_0(1-s)\leq r_0\sqrt{1-s},\quad r_0s-\tilde{r}\leq r_0-\tilde{r}\leq r_0\sqrt{1-s}.
  \end{align*}
Define the relation $\sim $ meaning that two quantities are of the same size up to some universal constant, that is, $A\sim B$ means $C^{-1}B\leq A\leq CB$ for some constant $C$. The above computation shows that for $0\leq s\leq s_0$
\begin{align*}
  r\sim |\tilde{r}-r_0s|+\sqrt{1-s^2}r_0\sim \sqrt{1-s}r_0.
\end{align*}
Moreover when $\tilde{r}\geq r_0s$, it trivially has
\begin{align*}
  r+\tilde{r}-r_0s\sim \sqrt{1-s}r_0.
\end{align*}
Otherwise by writing
\begin{align*}
  r+\tilde{r}-r_0s=\frac{r^2-(\tilde{r}-r_0s)^2}{r+r_0s-\tilde{r}}=\frac{(1-s^2)r_0^2}{r+r_0s-\tilde{r}}\sim\frac{(1-s^2)r_0^2}{\sqrt{1-s}r_0}\sim \sqrt{1-s}r_0.
\end{align*}
  Therefore on the interval $[0, s_0]$, we can estimate that
  \begin{align*}
    \int_{0}^{s_0}  (r^{\ga-1}(r+\tilde{r}-r_0s)+1)^{-p} ds
    &\les  \int_{0}^{s_0}(\sqrt{1-s} r_0)^{-p\ga} ds \les  r_0^{-p\ga}.
  \end{align*}
  Here we may note that $p\ga < p(p-1)\leq 2$.

  Finally on the interval $[s_0, 1]$, notice that
  \begin{align*}
    r_0s\geq r_0s_0\geq \tilde{r},\quad \sqrt{1-s}r_0\leq r_0-\tilde{r}.
  \end{align*}
  Therefore we have
  \begin{align*}
    r\sim r_0s-\tilde{r}+\sqrt{1-s}r_0=r_0-\tilde{r}+(\sqrt{1-s}-(1-s))r_0\sim r_0-\tilde{r}.
  \end{align*}
  Hence we can estimate that
  \begin{align*}
   r(1+\tau)= r+\tilde{r}-r_0s=\frac{(1-s^2)r_0^2}{r+r_0s-\tilde{r}}\sim \frac{(1-s)r_0^2}{r_0-\tilde{r}}.
  \end{align*}
This leads to the bound that
  \begin{align*}
    &\int_{s_0}^{1}  (r^{\ga-1}(r+\tilde{r}-r_0s)+1)^{-p} ds\\
    &\les  \int_{s_0}^{1}\big(1+ (r_0-\tilde{r})^{\ga-2}(1-s)r_0^2\big)^{-p} ds\\
    &\les (r_0-\tilde{r})^{2-\ga} r_0^{-2} .
    %&\leq C_\ep (r_0-\tilde{r})^{2-\b-\ga+\ep} r_0^{-2}\left((r_0-\tilde{r})^{(1-\a)\ga}+(r_0-t_0)^{(1-\a)\ga}\right)
  \end{align*}
  Combining the above estimates, we have shown that in the exterior region $\{t_0\leq |x_0|\}$
  \begin{align*}
    \int_{\S_{(t_0-\tilde{r}, x_0)}(\tilde{r})} ((1+\tau)r^{\ga}+1)^{-p} d\tilde{\om} \les r_0^{-p\ga}+(r_0-\tilde{r})^{2-\ga} r_0^{-2}\les (1+r_0+\tilde{r})^{-\ga}.
  \end{align*}
This means that the Lemma holds for the case when $t_0\leq r_0$.

\bigskip

In the following, we consider the situation when $t_0>r_0$ and $t_0+r_0>10$. The integral on $[-1, 0]$ is easy to control. Indeed, when $s\leq 0$, by the expression of $r$ and $\tau$, we have
\begin{align*}
  r\sim \tilde{r}+r_0,\quad r(1+\tau)\sim \tilde{r}+r_0.
\end{align*}
Therefore, we can show that
\begin{align*}
  \int_{-1}^{0}(1+r^{\ga}(1+\tau))^{-p}ds\les \int_{-1}^{0}(1+(\tilde{r}+r_0)^{\ga})^{-p}ds\les (1+\tilde{r}+r_0)^{-p\ga}.
\end{align*}
For the integral on the interval $[0, 1]$, when $\tilde{r}\geq 2r_0$, we have
\begin{align*}
  r\sim \tilde{r}-r_0s+\sqrt{1-s}r_0\sim \tilde{r},\quad
  r+\tilde{r}-r_0s\sim \tilde{r}.
\end{align*}
Since $r_0\leq \f12 \tilde{r}$, we can show that
\begin{align*}
  \int_0^{1}(1+r^{\ga}(1+\tau))^{-p}ds\les \int_0^{1}(1+\tilde{r}^{\ga})^{-p}ds\les (1+\tilde{r}+r_0)^{-p\ga}.
\end{align*}
Now for the case when $r_0\leq \tilde{r}\leq 2 r_0$, similarly $r$ and $r(1+\tau)$ behave like
\begin{align*}
  &r\sim \tilde{r}-r_0s+\sqrt{1-s}r_0\sim \tilde{r}-r_0+\sqrt{1-s}r_0,\\
  & r\leq r(1+\tau)=r+\tilde{r}-r_0s\leq 2r.
\end{align*}
This shows that
\begin{align*}
  \int_0^1 (1+r^{\ga}(1+\tau))^{-p}ds\les \int_0^{1}(1+\sqrt{1-s}r_0)^{-p\ga}ds\les (1+\tilde{r}+r_0)^{-p\ga}.
\end{align*}
Again here we used the assumption that $p\ga<2$ and $\tilde{r}\leq 2r_0$.

Finally when $0\leq \tilde{r}\leq r_0$, we split the integral on $[0, 1]$ into several parts. On $[0, r_0^{-1}\tilde{r}]$, similar to the above case when $r_0\leq \tilde{r}\leq 2r_0$, we can estimate that
\begin{align*}
  \int_0^{r_0^{-1}\tilde{r}} (1+r^{\ga}(1+\tau)^{\ga})^{-p}ds\les \int_0^{r_0^{-1}\tilde{r}}(1+\tilde{r}-r_0+\sqrt{1-s}r_0)^{-p\ga}ds\les (1+\tilde{r}+r_0)^{-p\ga}.
\end{align*}
On $[r_0^{-1}\tilde{r}, 1]$, firstly we have
\begin{align*}
  r&\sim r_0s-\tilde{r}+\sqrt{1-s}r_0,\\
   r(1+\tau)&=\frac{(1-s^2)r_0^2}{r+r_0s-\tilde{r}} \sim\frac{(1-s) r_0^2}{r_0s-\tilde{r}+\sqrt{1-s}r_0}.
\end{align*}
Therefore, we can bound that
\begin{align*}
  \int_{r_0^{-1}\tilde{r}}^{1} (1+r^{\ga}(1+\tau))^{-p}ds \les \int_{r_0^{-1}\tilde{r}}^{1}(1+(r_0s-\tilde{r}+\sqrt{1-s}r_0)^{\ga-2}(1-s)r_0^2)^{-p}ds.
\end{align*}
Notice that
\begin{align*}
\frac{1}{2}(r_0-\tilde{r}+\sqrt{1-s} r_0)\leq r_0s-\tilde{r}+\sqrt{1-s}r_0\leq r_0-\tilde{r}+\sqrt{1-s}r_0,\quad r_0^{-1}\tilde{r}\leq s\leq 1.
\end{align*}
Denote $s_*=1-(1-r_0^{-1}\tilde{r})^2$. In particular, we have
\begin{align*}
  r_0^{-1}\tilde{r}\leq s_*\leq 1.
\end{align*}
On the interval $[s_*, 1]$, we have
\begin{align*}
  \sqrt{1-s}r_0\leq r_0-\tilde{r}.
\end{align*}
Therefore we show that
\begin{align*}
  \int_{s_*}^{1}(1+(1+\tau)r^{\ga})^{-p}ds &\les \int_{s_*}^1(1+(r_0-\tilde{r})^{\ga-2}(1-s)r_0^2)^{-p}ds\\
  &\les (r_0-\tilde{r})^{2-\ga}r_0^{-2}(1-(1+(r_0-\tilde{r})^{\ga})^{1-p})\\
  &\les (1+r_0)^{-\ga}\\
  &\les (1+r_0+\tilde{r})^{-\ga}.
\end{align*}
Otherwise on $[r_0^{-1}\tilde{r}, s_*]$, we can bound that
\begin{align*}
  \int_{r_0^{-1}\tilde{r}}^{s_*}(1+(1+\tau)r^{\ga})^{-p}ds &\les \int^{s_*}_{r_0^{-1}\tilde{r}}(1+(1-s)^{\f12\ga}r_0^{\ga})^{-p}ds\\
  &\les \int^{s_*}_{r_0^{-1}\tilde{r}}(1+(1-s) r_0^{2})^{-\f12 p\ga}ds \\
  &\les r_0^{-2}(1+r_0(r_0-\tilde{r}))^{1-\f12 p\ga }\\
  &\les  (1+\tilde{r}+r_0)^{-p\ga}.
\end{align*}
Here keep in mind that $\tilde{r}\leq r_0$.
The Lemma holds by combining all the above bounds.
\end{proof}

We are now ready to prove the main Theorem \ref{thm:main}.
The proof for the uniform boundedness of solution to \eqref{eq:NLW:semi:3d} relies on the representation formula for linear wave equations. The nonlinearity will be controlled by using the weighted energy estimates in Proposition \ref{prop:EF:cone:NW:3d:smallp}. Note that for $q=(t_0, x_0)$, we have
\begin{equation}
\label{eq:rep4phi:ex}
\begin{split}
4\pi\phi(t_0, x_0)&=\int_{\tilde{\om}}t_0  \phi_1(x_0+t_0\tilde{\om})d\tilde{\om}+\pa_{t_0}\big(\int_{\tilde{\om}}t_0  \phi_0(x_0+t_0\tilde{\om})d\tilde{\om}   \big)-\int_{\mathcal{N}^{-}(q)}|\phi|^{p-1} \phi \ \tilde{r} d\tilde{r}d\tilde{\om}.
\end{split}
\end{equation}
The linear evolution part is uniformly bounded
\begin{align*}
  |\int_{\tilde{\om}}t_0  \phi_1(x_0+t_0\tilde{\om})d\tilde{\om}+\pa_{t_0}\big(\int_{\tilde{\om}}t_0  \phi_0(x_0+t_0\tilde{\om})d\tilde{\om}   \big)|
  %&\les \int_{S_{(0, x_0)}(t_0)} t_0 |\pa\phi| +|\phi| d\tilde{\om}\\
  &\les \sqrt{\mathcal{E}_{1, 0}[\phi] }.
\end{align*}
To control the nonlinearity, in view of Lemma \ref{lem:integration:smallp} and Proposition \ref{prop:EF:cone:NW:3d:smallp}, we bound that
\begin{align*}
  &|\int_{\mathcal{N}^{-}(q) }|\phi|^{p-1}\phi \ \tilde{r} d\tilde{r}d\tilde{\om}|\\
  &\les \left(\int_{\mathcal{N}^{-}(q)}((1+\tau)r^{\ga}+1)|\phi|^{p+1}\ \tilde{r}^{2} d\tilde{r}d\tilde{\om}\right)^{\frac{p}{p+1}}  \cdot \left(\int_{\mathcal{N}^{-}(q) }((1+\tau)r^{\ga}+1)^{-p} \tilde{r}^{1-p} d\tilde{r}d\tilde{\om}\right)^{\frac{1}{p+1}}\\
  &\les (\mathcal{E}_{0, \ga}[\phi])^{\frac{p}{p+1}} \left(\int_{0}^{t_0}  (1+\tilde{r})^{-\ga} \tilde{r}^{1-p} d\tilde{r} \right)^{\frac{1}{p+1}}\\
  &\les (\mathcal{E}_{0, p-1}[\phi])^{\frac{p}{p+1}} \left(1+(1+t_0)^{2-p-\ga}\right)^{\frac{1}{p+1}}\\
  &\les (\mathcal{E}_{0, p-1}[\phi])^{\frac{p}{p+1}}
 % &\les \left(r_0^{-p\ga+2-p+\ep}+r_0^{-\ga}t_0^{2-p+\ep}(r_0-t_0)^{(1-p)\ga}+r_0^{1-\ga}t_0^{1-p}(r_0-t_0)^{(1-p)\ga}\right)^{\frac{1}{p+1}}\\
 %&\les \left(r_0^{-p\ga+2-p+\ep}+r_0^{1-\ga}t_0^{1-p+\ep}(r_0-t_0)^{(1-p)\ga}\right)^{\frac{1}{p+1}}
\end{align*}
by choosing $\ga=2-p+\ep$ such that $0<\ep<2p-3$. This shows that
\begin{align*}
  |\phi|\les \mathcal{E}_{1, 0}[\phi]+(\mathcal{E}_{0, p-1}[\phi])^{\frac{p}{p+1}}.
\end{align*}
By our definition, the implicit constant relies only on $p$. Hence the uniform boundedness of the main Theorem \ref{thm:main} follows.

\bibliography{shiwu}{}
\bibliographystyle{plain}

\bigskip

Beijing International Center for Mathematical Research, Peking University,
Beijing, China

\textsl{Email}: shiwuyang@math.pku.edu.cn

 \end{document}